\documentclass[journal]{IEEEtran}

\usepackage[utf8]{inputenc} 
\usepackage[T1]{fontenc}
\usepackage{url}
\usepackage{ifthen}
\usepackage{cite}
\usepackage[cmex10]{amsmath} % Use the [cmex10] option to ensure complicance
\usepackage{comment}
\usepackage{graphicx}
\usepackage{amsmath}
\usepackage{array}
\usepackage{cite}
\usepackage{amsfonts}%
\usepackage{amssymb}%
\usepackage{graphicx}
\usepackage{authblk}
\usepackage[margin=2cm]{geometry}
\usepackage{subfig}%
\usepackage{color}
\usepackage{dsfont}
\usepackage{comment} 

\usepackage{algorithm}
\usepackage{algpseudocode}

\newtheorem{theorem}{Theorem}

\newtheorem{lemma}{Lemma}

\newtheorem{proposition}{Proposition}

\newenvironment{proof}[1][Proof]{\textbf{#1.} }{\ \rule{0.5em}{0.5em}}

\newcommand{\cN}{\mathcal{N}}

\newcommand{\cT}{\mathcal{T}}

\begin{document}
\title{
SwiftCache: Model-Based Learning for Dynamic Content Caching in CDNs}

\author{Bahman Abolhassani, Atilla Eryilmaz, Tom Hou\\

\thanks{
B.  Abolhassani is with the Bradly Department of Electrical and Computer Engineering, Virginia Tech, Blacksburg, VA 99202 (e-mail:abolhassani@vt.edu).

A. Eryilmaz is with the  Department of  Electrical and  Computer  Engineering, The Ohio State University,  Columbus,  OH  43210  USA  (e-mail:eryilmaz.2@osu.edu).

T. Hou is with the Bradly Department of Electrical and Computer Engineering, Virginia Tech, Blacksburg, VA 99202 (e-mail:hou@vt.edu).
}
}

\maketitle

%%%%%%
%% Abstract: 
%% If your paper is eligible for the student paper award, please add
%% the comment "THIS PAPER IS ELIGIBLE FOR THE STUDENT PAPER
%% AWARD." as a first line in the abstract. 
%% For the final version of the accepted paper, please do not forget
%% to remove this comment!
%%

\begin{abstract}
We introduce SwiftCache, a "fresh" learning-based caching framework designed for content distribution networks (CDNs) featuring distributed front-end local caches and a dynamic back-end database. Users prefer the most recent version of the dynamically updated content, while the local caches lack knowledge of item popularity and refresh rates. We first explore scenarios with requests arriving at a local cache following a Poisson process, whereby we prove that the optimal policy features a threshold-based structure with updates occurring solely at request arrivals. Leveraging these findings, SwiftCache is proposed as a model-based learning framework for dynamic content caching. The simulation demonstrates near-optimal cost for Poisson process arrivals and strong performance with limited cache sizes. For more general environments, we present a model-free Reinforcement Learning (RL) based caching policy without prior statistical assumptions. The model-based policy performs well compared to the model-free policy when the variance of interarrival times remains moderate. However, as the variance increases, RL slightly outperforms model-based learning at the cost of longer training times, and higher computational resource consumption. Model-based learning's adaptability to environmental changes without retraining positions it as a practical choice for dynamic network environments. Distributed edge caches can utilize this approach in a decentralized manner to effectively meet the evolving behaviors of users.
\end{abstract}

\section{Introduction}
\label{sec:Intro}
Content distribution networks (CDNs) are crucial for efficiently delivering content to users while improving overall performance and reducing network congestion \cite{abolhassani2021fresh, abolhassani2022fresh, buyya2008content, bartolini2003optimal}. In the era of 5G networks, caching at the wireless edge is employed to accelerate content downloads and enhance wireless network performance \cite{liu2016caching, yao2019mobile, abolhassani2023optimal}.

However, the emergence of new services like video on demand, augmented reality, social networking, and online gaming, which produce constantly changing data, has rendered traditional caching methods less efficient. The importance of content freshness has grown significantly due to the rapid increase in mobile devices and real-time applications like video on demand and live broadcast \cite{zhang2019price, abolhassani2021single}. Simply serving content to users is no longer sufficient, as users prefer the latest data versions, valuing fresh content the most \cite{shapiro1998information, abolhassani2022fresh}. Directly delivering the freshest content from the back-end database to users can congest the network and degrade the quality of experience (QoE). These challenges are amplified with the growth of cellular traffic in the network. Consequently, there is a pressing need to develop new caching strategies that consider content refresh characteristics and aging costs to ensure efficient dynamic content distribution.

The dynamic nature of content in CDNs necessitates the implementation of effective caching strategies to maintain the freshness of cached content \cite{zhang2018towards, yan2021learning, abolhassani2023optimal}. This is crucial to enable users to access the most up-to-date information, especially in rapidly evolving environments where content popularity, demand intensity, and refresh rates are constantly fluctuating \cite{li2016popularity, kam2017information}.

Numerous works study the dynamic content delivery in caching systems such as \cite{najm2016age, candan2001enabling, shi2002workload, li2003freshness, wu2019dynamic, kam2017information, xu2020proactive, gao2020design, abolhassani2020achieving, mehrizi2019popularity, kumar2020optimized, zhong2018two, zhang2020reinforcement, masood2020learning, azimi2018online, zhang2019learning, abad2019dynamic} and effective strategies have been proposed. 
Despite the traditional caching paradigms, for content with varying generation dynamics, regular updates of cached content are crucial \cite{amadeo2020caching, bastopcu2020maximizing}, preventing caches from becoming outdated due to aging content. The widespread deployment of edge caches necessitates a decentralized approach, with each cache independently managing its content freshness to avoid becoming outdated. These edge caches face the challenge of balancing update costs to ensure cost-effective operations. However, they often lack access to vital network parameters, requiring them to learn and adapt strategies for efficient cache management. Rapid changes in system parameters further demand continuous observation and dynamic adjustment of strategies based on user behavior to maintain overall cost-effectiveness

Machine Learning techniques have been widely utilized in network environments to achieve optimal caching performance \cite{sethumurugan2021designing, chang2018learn, shuja2021applying}. Notably, Reinforcement Learning (RL) proves to be a viable option for edge caches, allowing them to learn user behavior through observing and interacting with them \cite{shuja2021applying, zhu2018deep, jiang2019multi}. The rate at which users generate requests can significantly impact the time required for the RL policy to gather sufficient data and accurately learn the environment. Additionally, when the environment undergoes changes, such as shifts in popularity or refresh rates, retraining becomes necessary as the older policy may no longer be suitable for the new environment. This can demand substantial computational resources at the edge caches which may exceed their available capacity. These challenges highlight the need for adaptable and efficient learning algorithms that can quickly adjust to changing conditions and ensure optimal caching performance in dynamic network environments.

In this study, we tackle the challenging problem of caching dynamic content in a general setting where local caches lack crucial information about item popularity, demand intensity, and refresh rates. To address these limitations and create an adaptable caching framework, we propose a "fresh" learning-based approach. We begin by investigating a scenario where requests follow a Poisson process, revealing that, under such conditions, the optimal policy adopts a threshold-based structure with updates occurring solely at request arrivals. Building on insights from this optimal policy, we introduce SwiftCache, a model-based learning algorithm that efficiently manages limited cache capacity, eliminating resource-intensive training. This approach dynamically monitors user behavior, adapting its strategy without the need for wasteful retraining. Beyond Poisson processes, we explore Reinforcement Learning (RL) as a caching policy for more general environments. While RL adapts to diverse arrival patterns, it demands longer training times and higher computational resources compared to model-based learning. Additionally, RL policies lack adaptability to environmental changes, requiring retraining as user behavior evolves.

To assess the performance of our proposed framework, we conduct extensive simulations and evaluate its cost-effectiveness compared to the optimal caching policy. Our results demonstrate that our proposed model-based learning emerges as a solid caching framework, achieving near-optimal cost and exhibiting resilience to environmental changes without the need for extensive computational resources. 

The rest of the paper is organized as follows. In Section~\ref{sec:Sys_Model}, we present a practical caching model for efficient dynamic content delivery with minimized long-term average costs. Section~\ref{sec: optimal-policy} identifies optimal caching policies in Poisson Process-based scenarios, revealing threshold-based structures. Building on these insights, Section~\ref{sec:model-based} introduces SwiftCache as a model-based learning framework for dynamic content caching, achieving near-optimal cost without knowledge of item popularity and refresh rates. In Section~\ref{sec:Model_free}, we propose a model-free reinforcement learning and compare it to our proposed algorithm highlighting their respective strengths and trade-offs. And finally we conclude the paper in Section~ \ref{sec:conclusion}.

\section{System Model}
\label{sec:Sys_Model}

The hierarchical caching system depicted in Fig.~\ref{fig:system,model} is designed to serve a user population with dynamic content. Requests arrive at the local cache with an intensity represented by the rate parameter $\beta\geq0$. The size of data items is denoted as $(b_n)_{n=1}^N$, and the popularity profile is characterized by $\mathbf{p}=(p_n)_{n=1}^N$. Updates for each data item occur at an average rate of $\lambda_n\geq 0$ updates per second, represented by the vector $\boldsymbol{\lambda}=(\lambda_n)_{n=1}^N$. Notably, the local cache which has a limited cash capacity, lacks knowledge of the popularity profile $\mathbf{p}$ and refresh rate $\boldsymbol{\lambda}$.

\begin{figure}[t]
\centering
\includegraphics[width=0.5\textwidth]{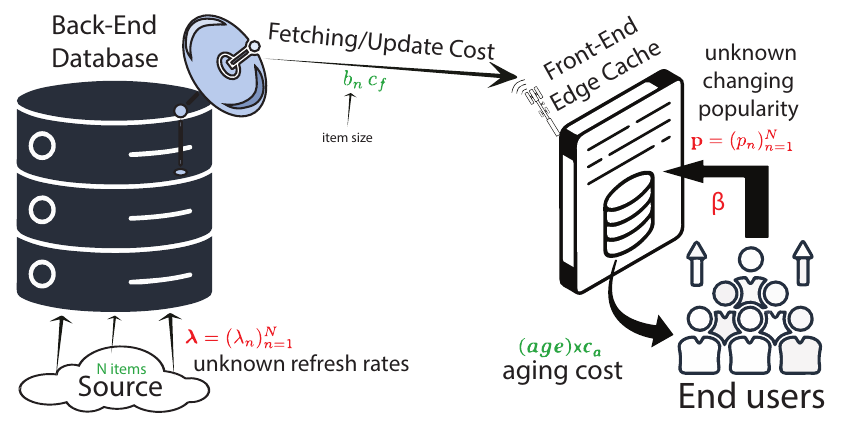}
\caption{Setting of \textit{Fresh Caching for Dynamic Content}}
\label{fig:system,model}
\end{figure}

\noindent\textbf{Age Dynamics: }To quantify freshness, we use the \emph{Age-of-Version} $\Delta_n(t)$, representing the number of updates since the item was last cached \cite{abolhassani2021fresh}.\\
\textbf{Fetching and Aging Costs: }Operational costs include fetching an item of size $b_n$ that incures cost $b_n \,c_f>0$. Performance costs account for serving an item with age $\Delta_n(t)$, incurring a freshness cost of $c_a \times \Delta_n(t)$.

\noindent \textbf{Problem Statement:} 
Our main goal is to develop efficient caching strategies that strike the right balance between the costs of frequently updating local content and providing aged content to users. Specifically, we seek a policy that minimizes the long-term average cost by making decisions on cache updates and request serving from the edge-cache. This objective can be formulated as follows:
\begin{equation}
\label{eq:costproblem}
\begin{split}
    &\min _{\boldsymbol{u}} \lim _{\mathrm{T} \rightarrow \infty} \frac{1}{T} \mathbb{E}_{s}^{u}\left[\int_{o}^{\mathrm{T}} \sum_{n=1}^N C_{n}^{\pi}\left(s_{n}(t),u_n(t)\right) dt\right]\\
    & s.t. \quad  u_n(t) \in \{0,1\}, \quad \forall n \in \cN , t\geq 0,
\end{split}
\end{equation}
where $u_n(t)\in \{0,1\}$ and $s_n(t)$ represent the action and state of the system at time $t$ for data item $n$, respectively. The cost of the system for data item $n$ at time $t$, given the state $s_n(t)$ and action $u_n(t)$ under a particular paradigm $\pi$, is denoted as $C_{n}^{\pi}\left(x_n(t),u_n(t)\right)$. The operator $\mathbb{E}_{s}^{u}$ captures the expected value over the system state $s$ for a given action $u$.
Our aim is to find the optimal policy $\boldsymbol{u}$ that minimizes this long-term average cost.

To address the challenges of the general problem expressed in \ref{eq:costproblem}, we first focus on a specific scenario where the arrival requests follow a Poisson Process. In addition, we assume that the local cache possesses information about item popularity, refresh rates, and demand intensity. These assumptions enable us to analyze the optimal policy's structure, which will be discussed in the next section.

\section{The Optimal Policy}
\label{sec: optimal-policy}

In this section, we focus on the previously introduced problem. Our goal is to analyze the structure of the optimal policy in a scenario where requests follow a Poisson process with a known rate $\beta$, and the local cache, constrained by limited storage capacity, is equipped with information about the popularity profile $\mathbf{p}$ and refresh rate $\boldsymbol{\lambda}$. Under these conditions, we present a clear characterization of the optimal policy in the following theorem.

\begin{theorem}
\label{thm:optimalpolicy}
Consider a system comprising a dataset $\mathcal{N}$ of $N$ dynamic items, where each item has a corresponding size denoted as $(\boldsymbol{b})_{n=1}^N$. In this system, a local cache with a constraint on its average cache occupancy $\tilde{B}$ serves a group of users that generate demand according to a Poisson Process with a known rate $\beta$. The local cache possesses knowledge of the popularity profile $\mathbf{p}=(p_n)_{n=1}^N$ and refresh rates $\boldsymbol{\lambda}=(\lambda_n)_{n=1}^N$. Under these conditions, the optimal caching strategy is an eviction-based policy. Specifically, the local cache assigns a constant timer to each item upon its storage, and the item is evicted from the cache when the timer expires. The timers for each item are determined as follows:
\begin{equation}
\label{eq:solutiont}
     \tau_n^* (\beta,\mathbf{p},\boldsymbol{\lambda})=\frac{1}{\beta p_{n}} \left[\sqrt{1+2 \, b_n \, \frac{\beta p_{n} c_{f}-\tilde{\alpha}^*}{c_{a} \lambda_{n}}}-1\right]^+,
\end{equation}
where $\tilde{\alpha}^* \geq 0$ is selected such that the following condition is satisfied.
\begin{equation}
\tilde{\alpha}^*\left(\sum_{n=1}^{N} \frac{\beta p_{n} \tau^*_{n}}{\beta p_{n} \tau^*_{n}+1} \, b_n-\tilde{B}\right)=0.
\end{equation}
\end{theorem}

\begin{proof}
    The outline of the proof is to first show that the optimal policy has a threshold-based structure, where updates happen only at the time of request arrivals. The subsequent step is to characterize the optimal thresholds. For the full proof, please refer to the appendix.
\end{proof}

As evident from Equation \ref{eq:solutiont}, the timers can be explicitly determined as a function of the system parameters, including the popularity, refresh rate, and size of the items. The parameter $\tilde{\alpha}$ is selected such that the constraint $\tilde{B}$ on the average cache occupancy is satisfied. The optimal value $\tilde{\alpha}^*$ exhibits an inverse relationship with $\tilde{B}$. As the cache capacity increases, $\tilde{\alpha}^*$ tends to approach zero. In the specific scenario of an unlimited cache capacity, this corresponds to setting $\tilde{\alpha}^* = 0$.

The following proposition provides a characterization of the optimal cost when the cache capacity is unlimited. We will utilize this cost as a benchmark for comparing the performance of our proposed algorithms. 

\begin{proposition}
\label{prop:optimal_unlimited}
Under the conditions outlined in Theorem \ref{thm:optimalpolicy}, in the scenario where the cache storage capacity is unlimited, the optimal cost can be expressed as follows:
\begin{equation}
\label{eq:optimalcost_unlimited}
C^*(\beta,\mathbf{p},\boldsymbol{\lambda})=\sum_{n =1}^N c_{0} \lambda_{n}\left(\sqrt{1+2 \, b_n \frac{\beta p_{n}}{\lambda_{n}} \frac{c_{f}}{c_{0}}}-1\right).
\end{equation}
\end{proposition}

\begin{proof}
According to Theorem 1, the optimal policy $\boldsymbol{\tau}$ has a threshold-based structure where cache updates happen only at the times of request arrivals. Under such a policy and based on the freshness and fetching cost defined, the average cost in Equation \ref{eq:costproblem} can be rewritten as:

\begin{equation*}
    C(\boldsymbol{\tau}) = \beta \sum_{n=1}^{N} p_{n} \frac{\frac{1}{2} c_{a} \lambda_{n} p_{n} \beta \tau_{n}^{2} + b_n c_{f}}{1+\beta p_{n} \tau_{n}}, \label{eq:cost}
\end{equation*}

Replacing the optimal threshold given in Equation \ref{eq:solutiont} under $\tilde{\alpha}^* = 0$, representing unlimited cache capacity, will result in the average cost given in Equation \ref{eq:optimalcost_unlimited}.
\end{proof}

While Proposition \ref{prop:optimal_unlimited} provides the optimal cost under unlimited cache capacity, in real-world scenarios, cache occupancy is inherently limited due to dynamic content. Periodic updates are necessary to minimize freshness costs. The optimal policy's eviction-based structure achieves this by briefly retaining each item before eviction, resulting in a decrease in cache occupancy as content becomes more dynamic.

The insights from Theorem \ref{thm:optimalpolicy} rely on the assumption that the local cache has complete knowledge of demand intensity, popularity profile, and refresh rates. In practice, such comprehensive knowledge is often unavailable, and these parameters can change. In the next section, leveraging the insights from Theorem \ref{thm:optimalpolicy}, we propose SwiftCache, a novel learning-based algorithm grounded in the eviction-based model. This algorithm eliminates the need for computationally demanding training procedures, saving computational resources. Additionally, its adaptability to dynamic system parameters ensures robustness in the face of changes, making it a practical and efficient solution.

\section{Model-Based Learning} 
\label{sec:model-based}
In this section, we focus on crafting an efficient caching policy for dynamic content, where requests follow a Poisson process. Despite lacking information on popularity, refresh rates, and demand intensity, our goal is to design a policy that adapts to changing environmental conditions. Our proposed dynamic caching policy continuously learns and adjusts to evolving popularity and refresh rates, minimizing long-term average costs and achieving near-optimal performance in dynamic content caching.

Drawing insights from Theorem \ref{thm:optimalpolicy}, we introduce SwiftCache, a model-based caching policy for dynamic content. Inspired by Theorem \ref{thm:optimalpolicy}, our proposed caching policy features an eviction-based structure, dynamically learning and adjusting timers to maintain average cache occupancy below a specified threshold while minimizing the average long-term costs.

\begin{algorithm}
\caption*{\textbf{SwiftCache: Model-Based Caching for Dynamic Content}}
\begin{algorithmic}[1]
\Require $c_f, c_0, N, \boldsymbol{b}, B, \theta$
\State \textit{Initialization}: $\alpha = 0, \boldsymbol{eit} = [0]^N, \hat{\boldsymbol{\lambda}} = [0]^N, \boldsymbol{\tau} =[0]^N$, 
  $ \Delta = [0]^N, \boldsymbol{\Bar{t}} = [0]^N, \boldsymbol{\Bar{s}} = [0]^N, \Bar{B} = 0$
\While{there is a request for an item} 
  \State $n = $ index of the requested item
  \State $t = \text{time} -  \boldsymbol{\Bar{t}}_n$
  \State $s = \text{time} - \boldsymbol{\Bar{s}}_n$
  \State $\boldsymbol{\Bar{s}}_n = \text{time}$
  \If{$t \geq \tau_n$} 
      \State Fetch item $n$ and read its age, $\delta$
      \State Calculate current cache occupancy $\hat{B}$
      \State $\boldsymbol{\Bar{t}}_n$ = \text{time}
      \State $\hat{\lambda}_n = (1-\theta) \hat{\lambda}_n + \theta \frac{\delta - \Delta_n}{t} $
      \State $\Delta_n = \delta $
      \State $\tau_n = eit_n \left[ \sqrt{1+2b_n \frac{c_f-eit_n\alpha}{c_0 \hat{\lambda_n} eit_n}}-1\right]^+$
  \EndIf
  \State $eit_n = (1-\theta) \, eit_n + \theta \, s$
  \State $\Bar{B} = (1-\theta) \, \Bar{B} + \theta \, \hat{B}$
  \State $ \alpha = \max \,(0, \, (\Bar{B} -B) / (\max_n (\boldsymbol{eit}))_{n=1}^N )$.
\EndWhile
\Return $\boldsymbol{\tau}$\hfill\textit{end}
\end{algorithmic}
\end{algorithm}

The proposed algorithm employs three parameters—last request time of items ($\boldsymbol{\Bar{s}}$), last fetch time of items ($\boldsymbol{\Bar{t}}$), and the corresponding age of items during the last fetch ($\boldsymbol{\Delta}$)—to effectively manage the caching process. By monitoring and utilizing these parameters, the algorithm optimizes the caching strategy for efficient content delivery. To achieve this, SwiftCache estimates refresh rates ($\hat{\boldsymbol{\lambda}}$) and interarrival times ($\boldsymbol{eit}$), utilizing them to calculate timers ($\boldsymbol{\tau}$) for each item. Specifically, the algorithm calculates the holding time $\tau_n$ for data item $n$ based on the estimated refresh rate $\hat{\lambda}_n$ and estimated interarrival times $eit_n$ as given by:
\begin{equation}
    \tau_n = eit_n * \left[ \sqrt{1+2b_n \frac{c_f-eit_n*\alpha}{c_0 \hat{\lambda}_n eit_n}}-1\right]^+,
\end{equation}

This calculation is rooted in the optimal policy provided in Theorem \ref{thm:optimalpolicy}. Additionally, the algorithm dynamically adjusts timer values to minimize costs while adhering to the cache occupancy constraint. This adaptive approach ensures an optimized caching strategy, even in the presence of fluctuations in popularity, refresh rates, and demand intensity. By actively adapting to these changes, the algorithm effectively handles the dynamic nature of the system, ensuring efficient and robust performance.

To assess our algorithm's performance against the benchmark cost outlined in Proposition \ref{prop:optimal_unlimited}, we consider a system with $N=1000$ items, each with a size of $b=10$ and refresh rates $\lambda = 20$. User requests for items follow a Zipf popularity distribution with a parameter of $z=1$, where $p_n$ is proportional to $\frac{1}{n}$. Simulations are conducted under various average cache occupancy constraints, denoted as ${B}$. Fetching cost per item size is represented by $c_f=1$, and aging cost per single age is $c_a = 0.1$. The time averaging parameter $\theta$ is set to $\theta = 0.005$. These parameters will be used throughout the paper unless specified otherwise.

We gauge performance using the percentage cost increase of our proposed model-based learning policy compared to the optimal caching's cost. The percentage cost increase is calculated using:
\begin{equation}
\label{eq:percentageMB}
\text{Percentage Cost Increase$(\%)$}=100 \times \frac{C^{MB}-C^*(\beta,\mathbf{p},\boldsymbol{\lambda})}{C^{MB}},
\end{equation}
where $C^{MB}$ is the average long-term cost of the model-based learning policy, and $C^*(\beta,\mathbf{p},\boldsymbol{\lambda})$ denotes the optimal average cost given in Equation \ref{eq:optimalcost_unlimited}, considering unlimited cache capacity and perfect knowledge of the system parameters. Evaluating the cost increase as a percentage allows for effective assessment of the proposed caching policy's performance in terms of cost optimization.

\begin{figure}[t]
\centering
\includegraphics[width=0.5\textwidth]{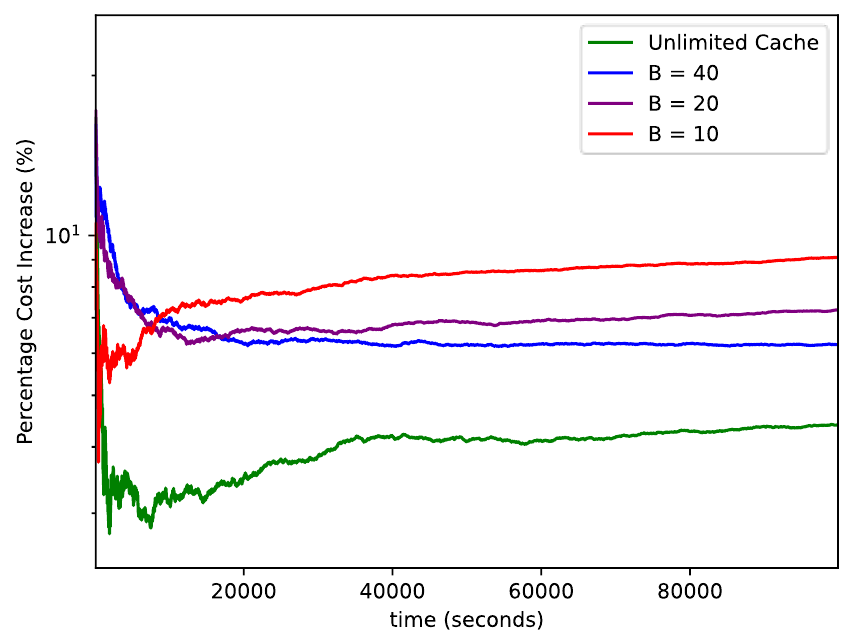}
\caption{Cost Comparison: SwiftCache vs. Optimal Policy}
\label{fig:model-based}
\vspace{-.1in}
\end{figure}

According to Fig. \ref{fig:model-based}, when the cache capacity is unlimited, our proposed caching policy demonstrates a mere 4$\%$ cost increase, approaching optimal performance. This remarkable performance is notable considering that the proposed policy does not possess knowledge of the popularity profile, refresh rates, and demand intensity, unlike the optimal policy.

Furthermore, as the cache size becomes limited, the performance gap between the proposed policy and the optimal strategy slightly increases. However, even in these constrained scenarios, the cost increase remains relatively low, below 10$\%$. As cache size becomes limited, the performance gap slightly widens; however, even in constrained scenarios, the cost increase stays below 10$\%$. This underscores the proposed policy's ability to not only learn popularity and refresh rates but also effectively utilize limited cache capacity, achieving lower average costs.

Our analysis, relying on Poisson process arrivals, led to the development of a model-based caching policy based on Theorem \ref{thm:optimalpolicy}. However, practical scenarios may deviate from the assumption of exponentially distributed interarrival times. To address this, we propose a model-free reinforcement learning approach. This approach allows the system to learn and utilize an optimal policy, regardless of the request arrival distribution. We seek to compare the performance of our model-based policy with the model-free approach in scenarios with diverse request arrival distributions, providing insights into their adaptability and effectiveness.

\section{Comparison to Model-Free Learning} 
\label{sec:Model_free}
In this section, we extend our analysis beyond the Poisson Process assumption, addressing dynamic content caching in a more general context. We consider scenarios where arrival requests don't follow a known distribution, and both update rate and popularity profiles are unknown. To tackle these challenges, we propose a learning approach for caching dynamic content based on Reinforcement Learning, aiming to enhance the efficiency of content delivery in such dynamic environments.

More specifically, we take the Q-Learning approach where the evolution of $Q$'s for each data item $n$ can be expressed as:
\begin{equation}
\begin{split}
& Q\left(S_t, u_t\right) \leftarrow Q\left(S_t, u_t\right) \\&+\alpha\left[R_{t+1}+\gamma \max _u Q\left(S_{t+1}, u\right)-Q\left(S_t, u_t\right)\right] .
\end{split}
\end{equation}

In our analysis, we utilize the notations $S_t$, $S_{t+1}$, and $u_t$ to represent the current state, next state, and the action taken at time $t$, respectively. Due to the unavailability of the probability distribution $P(S_{t+1} | S_y, u_t)$, the learning approach we employ is considered model-free. At each time step, there are two potential actions: cache update/fetch, denoted as $u_t = 1$, and no cache update, represented by $u_t = 0$. Using the estimated value for $Q$'s, the optimal action at each state $S_t$ is given by:
\begin{equation*}
    u_t^* = \underset{u \in \{0 , 1\}} {argmin} \; Q(S_t,u).
\end{equation*}

Given that the state $S_t$ evolves continuously over time $[0, \infty)$, we discretize time into steps of size $0.1$ seconds before applying the Q-Learning algorithm.

For a more insightful comparison between the proposed model-based learning (SwiftCache) and the model-free learning approach, we consider interarrival times of each data item $n$ following a Gamma distribution with parameters $(\omega, \frac{1}{\omega\beta p_n })$. This distribution ensures a constant average interarrival time of $\frac{1}{\beta p_n}$, regardless of $\omega$. The parameter $\omega > 0$ determines the variance, given by $\frac{1}{\omega (\beta p_n)^2}$. When $\omega = 1$, the distribution is equivalent to the exponential distribution. This choice facilitates a meaningful performance comparison between the proposed algorithms.

As a performance metric, we calculate the percentage cost increase of the model-free RL policy compared to the model-based caching policy (SwiftCache) using the formula:
\begin{equation}
\text{Percentage Cost Increase}(\%) = 100 \times \frac{C^{MF} - C^{MB}}{C^{MB}},
\end{equation}

This metric facilitates an effective comparison of their cost optimization capabilities. Positive values favor SwiftCache, while negative values favor the RL algorithm. We consider a system with $N=1000$ items, each of size $b=1$, requested uniformly with $p_n=0.001, \forall n$. Refresh rates are $\lambda = 20$. Interarrival times follow a Gamma distribution with parameters $(\omega, \frac{1}{\omega\beta p_n})$, where $\beta = 5$. Varying $\omega$ alters the interarrival time distribution, allowing us to analyze the impact on the performance of the caching algorithms and gain insights into their behavior under different scenarios.

\begin{figure}[t]
\centering
\includegraphics[width=0.5\textwidth]{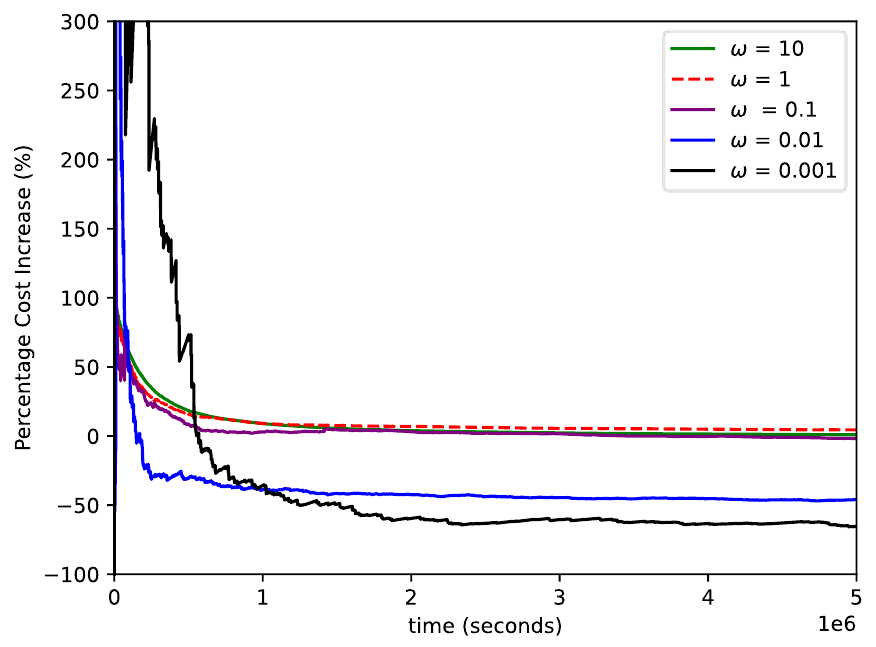}
\caption{Cost Comparison: SwiftCache vs. RL Policy}
\label{fig:comparison}
\end{figure}

According to Fig. \ref{fig:comparison}, for $\omega=1$, representing exponential interarrival times, SwiftCache slightly outperforms RL caching. This result was expected since SwiftCache is specifically designed for this type of environment.
Moreover, model-based learning, with faster convergence and less training time, proves competitive to RL caching in scenarios with modest variance. This emphasizes the cost-effectiveness and adaptability of model-based learning, requiring no retraining for environmental changes compared to the resource-intensive retraining needed for RL caching. On the other hand, with significantly reduced $\omega$ (e.g., $\omega = 0.001$), corresponding to large variances, RL caching outperforms model-based caching by almost 50$\%$. However, this improvement comes at the expense of extensive training time.

\begin{figure}[t]
\centering
\includegraphics[width=0.5\textwidth]{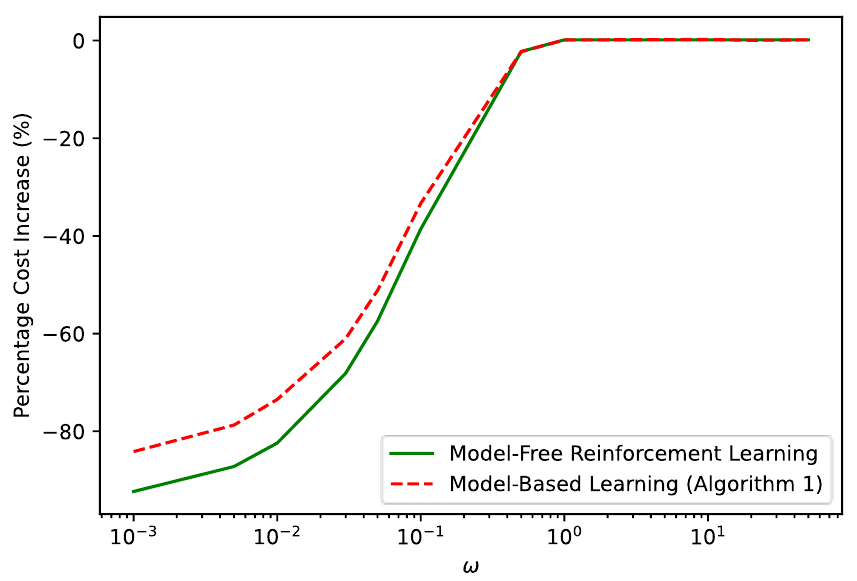}
\caption{Cost Comparison: SwiftCache vs. RL Policy}
\label{fig:omega}
\end{figure}

In Fig. \ref{fig:omega}, we compare the percentage cost increase of model-free reinforcement learning (given by Equation \ref{eq:percentageRL}) with the model-based policy (given by Equation \ref{eq:percentageMB}) against the optimal solution under the Poisson process assumption. The plot shows these increases as functions of the environment parameter $\omega$ over a $10^6$ seconds duration. Surprisingly, model-based learning, designed for Poisson arrivals, performs nearly as well as model-free RL in more general environments, maintaining high gains as the variance of interarrival times decreases. However, when the variance increases, model-free RL outperforms model-based learning. It's crucial to note that the higher gains of model-free RL come with longer training times and sensitivity to environmental changes, requiring retraining and significant computational resources.

\section{Conclusion}
\label{sec:conclusion}
SwiftCache stands out as an innovative model-based learning framework tailored for dynamic content caching in content distribution networks (CDNs). Addressing the challenge of local caches lacking knowledge of item popularity and refresh rates, SwiftCache is an efficient policy in scenarios with Poisson process arrivals. Leveraging a threshold-based optimal policy, SwiftCache achieves near-optimal cost without the need for resource-intensive training. While comparing favorably to a model-free Reinforcement Learning (RL) policy in environments with moderate variance, SwiftCache's key strength lies in its adaptability to environmental changes without retraining. This adaptability, coupled with its efficient performance, positions SwiftCache as a practical and resilient solution for dynamic network environments, offering effective content delivery that aligns with evolving user behaviors.

\appendix
\textbf{Proof of Theorem 1:}
We initiate the proof by establishing that the optimal policy for the Equation \ref{eq:costproblem}, assuming arrivals following Poisson processes, possesses a threshold structure. We then
explicitly characterize these thresholds.

\begin{equation}
\label{eq:costproblem}
\begin{split}
    &\min _{\boldsymbol{u}} \lim _{\mathrm{T} \rightarrow \infty} \frac{1}{T} \mathbb{E}_{s}^{u}\left[\int_{0}^{\mathrm{T}} \sum_{n=1}^N C_{n}^{\pi}\left(s_{n}(t),u_n(t)\right) dt\right]\\
    & \text{s.t.} \quad  u_n(t) \in \{0,1\}, \quad \forall n \in \cN, \, t \geq 0,
\end{split}
\end{equation}

To establish the threshold-based structure of the optimal policy, we initially relax the assumption of limited cache capacity and demonstrate the threshold-based structure under unlimited cache capacity.

When the cache size is unlimited, the average cost for each data item $n \in \cN$ becomes independent of all other items. Consequently, in Equation \ref{eq:costproblem}, we can move the summation to the outside. By utilizing the fetching and freshness costs defined in the paper, we express the optimization problem as:
\begin{equation} \label{eq:pullcontinous}
\begin{split}
    &\hspace{-.14in}  \sum_{n=1}^N \min _{\boldsymbol{u_n}}  \lim _{\mathrm{T} \rightarrow \infty} \frac{1}{T} \int_{0}^{\mathrm{T}}  \left[ u_n(t) b_n c_f \right. \\
    &\left. + r_n(t) s_n(t) \lambda_n c_a (1-u_n(t))\right]dt,
\end{split}
\end{equation}
where $s_n(t)$ denotes the system state and is defined as $s_n(t) = \{t - l : l = \max (t' \leq t : u_n(t') = 1)\} \geq 0$ for all $n \in \cN$. This expression represents the elapsed time since the last update of item $n$ in the cache.

Hence, by decoupling the optimization over items, we can concentrate on optimizing the cost for each data item $n$ separately. Additionally, we define $r_n(t) \in \{0,1\}$ to signify the event of item $n$ being requested at time $t$, i.e., $r_n(t) = 1$ if there is a request at time $t$, and $r_n(t) = 0$ in the case of no request.

In our approach, we adopt the vanishing discounted cost approach to address the average cost problem. Consider the following optimization problem:

\begin{equation} \label{eq:costpullvanish}
%\vspace{-.1in}
\begin{split}
    &\min _{\boldsymbol{u_n}} \lim _{\mathrm{T} \rightarrow \infty}  \mathbb{E}_{x}^{u} \int_{0}^{\mathrm{T}} e^{-\alpha t} C(x_n(t),u_n(t))dt\\
    & \text{s.t.} \quad  u_n(t) \in \{0,1\}, \quad \forall t\geq 0,\\
    & \hspace{.32in}  x_n(t)=(s_n(t),r_n(t)), \quad r_n(t)\in \{0,1\}, \quad s_n(t)\geq 0,\\
    & \hspace{.32in}  s_n(t+dt)=s_n(t)(1-u_n(t))+dt, \quad \forall n \in \cN , \, t\geq 0.
\end{split}
\end{equation}

In the next proposition, we demonstrate that by focusing on the decision only at the jump times of $r$, the average cost cannot increase.

\begin{proposition} \label{prop:pulljump}
Assume $\boldsymbol{u_n}$ is an arbitrary policy. By modifying the policy $\boldsymbol{u_n}$ such that each action $u_n(t) = 1$ will wait until the next request, i.e., $r_n(t) = 1$, we define a new policy $\boldsymbol{u'_n}$ as a modified version of $\boldsymbol{u_n}$. The policy $\boldsymbol{u'_n}$ will not result in a higher average cost than $\boldsymbol{u_n}$ for the optimization problem \eqref{eq:pullcontinous}.
\end{proposition}

\begin{proof} 
Consider the policy $\boldsymbol{u}$, where there exists a time $t$ such that $r(t) = 0$ and $u(t) = 1$. In other words, $\boldsymbol{u}$ updates the cache when there is no request. Let $C^{\boldsymbol{u}}$ be the corresponding cost. Let $t' = \min \left\{\tau > t \,|\, r(\tau) = 1 \right\}$ be the first arriving request after time $t$. We construct a new policy $\boldsymbol{u'}$ by setting $u'(\tau) = 0$ for $t \leq \tau < t'$ and $u'(t') = 1$. Now, we show that this modification does not increase the average cost of Equation \eqref{eq:pullcontinous}.

Since the two policies are identical outside the interval $[t, t')$, we focus on this time period. Policy $\boldsymbol{u}$ updates the cache at time $t$ with a cost of $c_f$. When the next request arrives at $t'$, the cached content might be old, incurring an average freshness cost of $c_a \lambda_n (t' - t)$. Thus, $C^{\boldsymbol{u}}([t, t']) = b_n c_f + c_a \lambda_n (t' - t)$. In contrast, policy $\boldsymbol{u'}$ updates the cache only when there is a request, eliminating potential freshness cost. Therefore, $C^{\boldsymbol{u'}}([t, t']) = c_f$. Hence, by this slight modification to $\boldsymbol{u}$, the cost cannot increase.
\end{proof}

If we consistently apply the policy from Proposition \ref{prop:pulljump} to every update instance of policy $\boldsymbol{u}$—ensuring that cache updates only happen at times of request arrivals—the result is a policy that updates the cache only at instances when there is a request. Such a policy will not perform worse than the original policy $\boldsymbol{u}$. Therefore, in line with Proposition \ref{prop:pulljump} and following the approach in Appendix A, we can view the cost as the cost over $m$ time steps for a discrete-time decision process with a discount factor $q$. Here, the time steps correspond to the jump times of $r_n$, which is a rate $\beta p_n$ Poisson process. Thus, we focus on optimizing the discrete-time vanishing discounted cost problem stated below:

\begin{equation} \label{eq:pulldiscount}
%\vspace{-.15in}
\begin{split}
    &\min _{\boldsymbol{u}} \lim _{m \rightarrow \infty}  \mathbb{E}_{x}^{u} \sum_{k=0}^{m-1} q^k C(x_{t_k}),
\end{split}
\end{equation}

where $0 < q < 1$ is the discount factor, and $x_{t_k}=(s_{t_k},r_{t_k})$ represents the state at time step $k$. Since the cost function is non-negative, following the value iteration algorithm, we define:

\begin{equation*}
%\vspace{-.08in}
    v_{m}(x)=\min _{u_{n}}\left\{c(x, u)+q \sum_{y} v_{m-1}(y) \mathrm{p}(y \mid x, u)\right\}, \quad \forall x, m,
\end{equation*}

where $x=(s,r)$ is the state of the system, and $u \in \{0,1\}$ is the action taken. Additionally, $v_0(x)=0$ for all $x$. Note that $v_m$ is a monotonically non-decreasing sequence.

\begin{lemma}
$v_{m}(s, r)$, for all $m \geq 1$ and $r \in \{0,1\}$, is a non-decreasing function of $s$.
\end{lemma}

\begin{proof}(Lemma 1): We use induction. According to the definition of $v_m(s,r)$, we can show that:

\begin{equation*}
%\vspace{-.06in}
    v_{1}(s, r) = \begin{cases}r b_n c_{f} & s>b_n c_{f} / \lambda_n c_{a}, \\ \lambda_n c_{a} r s & s \leq b_n c_{f} / \lambda_n c_{a},\end{cases}
\end{equation*}

which is a non-decreasing function of $s$. Assume $v_{m-1}(s,r)$ is non-decreasing over $s$, then we have:

\begin{equation*}
%\vspace{-.05in}
\begin{split}
    v_{m}(s, r)=\min & \{\lambda_n c_{a} s r+q\beta p_n v_{m-1}(s+1,1), \\
    & b_n c_{f}+q\beta p_n v_{m-1}(1,1)\}.
\end{split}
\end{equation*}

Therefore, $v_{m}(s, r)$, for all $m \geq 1$ and $r \in \{0,1\}$, being the minimum of two non-decreasing functions of $s$, is itself a non-decreasing function of $s$.\end{proof}

\begin{proposition} \label{prop:pulloptimalr}
For $r=1$, there exists a $\tau_m$ such that the policy that minimizes $v_m(s,1)$ is given by:

\begin{equation} 
%\vspace{-.05in}
\label{eq:u(s,1)}
    u_{m}^{*}(s, 1)= \begin{cases}1 & s>\tau_{m} \\ 0 & s \leq \tau_{m}\end{cases}.
\end{equation}
\end{proposition}

\begin{proof}(Proposition \ref{prop:pulloptimalr}): For $m=1$, we have $\tau_{1}=b_n c_{f} / \lambda_n c_{a}$, and the assumption holds. According to the definition of $v_m(s,r)$ and focusing on the jump times of $r_n$, we have:

\begin{equation*} 
%\vspace{-.05in}
\begin{split}
    \hspace{-.1in} v_{m}(s, 1)=\min  \{& \underbrace{  \lambda_n c_{a} s +q v_{m-1}(s+1,1)}_{u_{m}=0}, \\
    & \underbrace{b_n c_{f}+q v_{m-1}(1,1)}_{u_{n}=1}\},
    \end{split}
\end{equation*}

where the first term is the projected cost when the request is served from the edge-cache (i.e., $u_m=0$), and the second term is the projected cost of fetching the request from the back-end database. Moreover, according to Lemma 1, $v_{m-1}(s,r)$ is a non-decreasing function of $s$. Therefore, if for a given $s$, $u_m^*(s,r)=1$, then for all $s'>s$ we have $u_m^*(s',r)=1$. Besides, if for a given $s$, $u_m^*(s,r)=0$, then for all $s'>s$ we have $u_m^*(s',r)=0$. Thus, there exists a threshold $\tau_m$ where Equation \eqref{eq:u(s,1)} holds.\end{proof}

According to Propositions \ref{prop:pulljump} and \ref{prop:pulloptimalr}, there exists a threshold $\tau_m$ where the optimal action $u_m^*(s,r)$ that minimizes $v_m(x)$, for all $m$, can be expressed as:

\begin{equation} 
\label{eq:pullhypothesis} 
u_{m}^{*}(s, r)= \begin{cases}r & s>\tau_{m} \\ 0 & s \leq \tau_{m}\end{cases}.
\end{equation}

Next, we demonstrate that the sequence $\{v_m(x)\}$ converges to $v(x)$ as $m$ increases. Since $\{v_m\}$ is a monotonically non-decreasing sequence, we establish that there exists a policy such that as $m$ increases, the cost remains finite. Assume that $\tau$ is constant. Consider the policy $\bar{u}_{\tau}=\bar{u}_{2\tau}=\bar{u}_{3\tau}=\ldots=1$, where we only fetch every $\tau$ time slot. Under this policy, we have $(1-q^\tau)v_{\infty}^{\bar{u}}=g(q)<\infty$, where $g(q)=\beta \lambda \sum_{n=1}^{\tau-1}nq^n+q^\tau b_n c_f$ is always finite for a given $\tau$ and $q<1$. Therefore, $v_m(x)\leq v_{m+1}(x)\leq \ldots \leq v_{\infty}^{\bar{u}}$. According to the Monotone Convergence Theorem, there exists $v(x)$ such that $v_{m}(x) \stackrel{m \rightarrow \infty}{\longrightarrow} v(x)$. Thus, the optimal policy that minimizes $v(x)$ is obtained through $\lim_{m\to \infty } u_{m}^{*}(s, r) =u^{*}(s, r)$ and is given by:

\begin{equation*}
%\vspace{-.09in}
u^{*}(s, r)= \begin{cases}r & s>\tau^*\\ 0 & s \leq \tau^* \end{cases},
\end{equation*}

where $\tau^*= \lim_{m\to \infty } \tau_m \leq \frac{b_n c_f}{\lambda c_a}$, since $\tau_1=\frac{b_n c_f}{\lambda c_a}$ for all $q \in (0,1)$ and $0 \leq \tau_{m}\leq \tau _1$ for all $m$. This follows from Lemma 1, which states that for any given $q \in (0,1)$, we have $v_{m-1}(s+1,1)\geq v_{m-1}(1,1)$ for all $m$ and $s\geq 0$. For the average cost setup, we now argue that there exists a solution to the average cost minimization problem, and this solution admits an optimal control of threshold type.

The existence of a solution to the average cost optimality follows from the moment nature of the cost function and the derivation of a finite cost solution using the convex analytic method \cite{Borkar2}. The optimality inequality under a stable control policy is established via \cite[Prop. 5.2]{hernandez1993existence}, and further discussions on both existence and optimality inequality can be found in \cite[Theorem 2.1 and Section 3.2.1]{arapostathis2021optimality}.

Consider the relative value functions solving the discounted cost optimality equation for each $q \in (0,1)$:

\begin{equation}
    \begin{split}
& J_q(x) - J_q(x_0) = \\& \min_u \left(c(x,u) + q \int (J_q(y) - J_q(x_0))p(dy|x,u)\right)      
    \end{split}
\end{equation}
Via the threshold optimality of finite horizon problems, since the threshold policies, $\tau_q$, are uniformly bounded, there exists a converging subsequence such that $\tau_{q_m} \to \tau^*$. Consequently,

\begin{equation}
    \begin{split}
& \lim_{m \to \infty} J_{q_m}(x) - J_{q_m}(x_0) = \\& \lim_{m \to \infty}  \min_u \left(c(x,u) + q_m \int (J_{q_m}(y) - J_{q_m}(x_0))p(dy|x,u)\right) \\& = \lim_{m \to \infty} \left( c(x,u_{\tau_{q_m}}) + q_m \int (J_{q_m}(y) - J_{q_m}(x_0))p(dy|x,u_{\tau_{q_m}}) \right)    
    \end{split}
\end{equation}

By considering a converging sequence and applying \cite[Theorem 5.4.3]{WinNT}, we conclude that the average cost optimality inequality applies under the limit threshold policy. Thus, by relating the Discounted Cost Optimality Equation (DCOE) and Average Cost Optimality Inequality (ACOI), the optimality of threshold policies follows.

Now, using the obtained result, the optimal policy for the average cost problem, with a threshold structure, allows us to determine the optimal threshold and the subsequent optimal cost.

In the following, we characterize the optimal thresholds for the optimal policy under Poisson process arrivals.

Given that the local cache has a constraint on its average cache occupancy, denoted by $\tilde{B}$, our goal is to solve the following average-cache-constrained problem:

\begin{eqnarray}
    &\displaystyle \min_{\boldsymbol{\tau}\in \cT}&  C (\boldsymbol{\tau} )\label{eq:averageconstmin}
 \\
    & s.t.& B (\boldsymbol{\tau} )\leq \tilde{B}, \nonumber 
\end{eqnarray}

where $C(\boldsymbol{\tau})$ and $B(\boldsymbol{\tau})$ are the average cost and average cache occupancy under the threshold-based policy $\boldsymbol{\tau}$ and are given by:

\begin{eqnarray}
    C(\boldsymbol{\tau}) &=& \beta \sum_{n=1}^{N} p_{n} \frac{\frac{1}{2} c_{a} \lambda_{n} p_{n} \beta \tau_{n}^{2} + b_n c_{f}}{1+\beta p_{n} \tau_{n}}, \label{eq:cost} \\
    B(\boldsymbol{\tau}) &=& \sum_{n=1}^{N} \frac{\beta p_n \tau_n}{1+\beta p_n \tau_n}, \label{eq:cacheoc}
\end{eqnarray}

We note that this problem is non-convex since the constraint set $\{\boldsymbol{\tau}: B(\boldsymbol{\tau}) \leq \tilde{B}\}$ is non-convex. However, we take the following approach to solve it. Define the feasible set $\mathcal{F}_{B}$ as:

\begin{equation*}
    \mathcal{F}_{B} = \left\{\boldsymbol{\tau} | \tau_{n} \geq 0, g(\boldsymbol{\tau})=\sum_{n=1}^{N} \frac{\beta p_{n} \tau_{n}}{\beta p_{n} \tau_{n}+1} \leq \tilde{B}\right\},
\end{equation*}

which is a non-convex set. Then the cost optimization problem \eqref{eq:averageconstmin} can be expressed as:

\begin{equation}
\label{eq:costminconstraintfining}
    \min_{\boldsymbol{\tau} \in \mathcal{F}_{B}} C(\boldsymbol{\tau}).
\end{equation}

For any optimization problem $\min _{\boldsymbol{\tau} \in \mathcal{F}} C(\boldsymbol{\tau})$ as presented in \cite{ho2017necessary}, if all the following conditions hold:

1) Slater condition,
   
2) Non-degeneracy assumption for $\forall \boldsymbol{\tau} \in \mathcal{F}$,

3) $\exists \boldsymbol{\tau}^{\prime} \in \mathcal{F}: \forall \boldsymbol{\tau} \in \mathcal{F}, \exists \boldsymbol{t}_{n} \downarrow 0$ with $\boldsymbol{\tau}^{\prime}+\boldsymbol{t}_{n}(\boldsymbol{\tau}-\boldsymbol{\tau}^{\prime}) \in \mathcal{F}$,

4) $L_{C}(\boldsymbol{\tau})=\left\{\boldsymbol{\tau}^{\prime} \in \mathbb{R}^{N}: C\left(\boldsymbol{\tau}^{\prime}\right)<C(\boldsymbol{\tau})\right\}$ is a convex set,

then if $\boldsymbol{\tau}$ is a non-trivial KKT point, it is a global minimizer.

In our case, all four conditions above hold. Therefore, to obtain the global minimizer, we need to write the KKT conditions. The final solution is given by:

\begin{equation}
\label{eq:solutiont}
    \tau_n^* (\beta,\mathbf{p},\boldsymbol{\lambda})=\frac{1}{\beta p_{n}} \left[\sqrt{1+2 \, b_n \, \frac{\beta p_{n} c_{f}-\tilde{\alpha}^*}{c_{a} \lambda_{n}}}-1\right]^+,
\end{equation}

where $\tilde{\alpha}^* \geq 0$ is selected such that the following KKT condition is satisfied:

\begin{equation}
\tilde{\alpha}^*\left(\sum_{n=1}^{N} \frac{\beta p_{n} \tau^*_{n}}{\beta p_{n} \tau^*_{n}+1} \, b_n-\tilde{B}\right)=0.
\end{equation}

This completes the proof of Theorem 1.

\bibliographystyle{IEEEtran}
\bibliography{ref}

\end{document}